%% file: main.tex
\pdfoutput=1
\documentclass[12pt,a4paper,reqno]{amsart}

\calclayout

\usepackage[utf8]{inputenc}
    \usepackage[T1]{fontenc}
    \usepackage{libertine}
    \usepackage[libertine,cmintegrals,cmbraces]{newtxmath}
    
    \usepackage{stackengine,amsmath,amsthm,amssymb,mathtools}
    \usepackage{mathdots}
    \usepackage{graphicx}
    \usepackage{tikz,tikz-cd}
            \tikzset{
                subseteq/.style={
                draw=none,
                edge node={node [sloped, allow upside down, auto=false]{$\subseteq$}}},
                Subseteq/.style={
                draw=none,
                every to/.append style={
                edge node={node [sloped, allow upside down, auto=false]{$\subseteq$}}}
                }
            }
            \tikzset{
                equal/.style={
                draw=none,
                edge node={node [sloped, allow upside down, auto=false]{$=$}}},
                Equal/.style={
                draw=none,
                every to/.append style={
                edge node={node [sloped, allow upside down, auto=false]{$=$}}}
                }
            }
            \tikzset{
                isom/.style={
                draw=none,
                edge node={node [sloped, allow upside down, auto=false]{$\cong$}}},
                Isom/.style={
                draw=none,
                every to/.append style={
                edge node={node [sloped, allow upside down, auto=false]{$\cong$}}}
                }
            }
            \tikzset{
            rotated_label/.style={anchor=south, rotate=90, inner sep=.5mm}
            }
    \usetikzlibrary{decorations}
    \usepackage{verbatim}
    \usepackage{booktabs}
    \usepackage{array}
    \usepackage{tabularx}
    \usepackage{booktabs}
    \usepackage{siunitx}
    \usepackage{rotating}
    \usepackage{adjustbox}
    \usepackage{longtable}
    \usepackage{caption} 
    \usepackage{stmaryrd}
    \usepackage{nicefrac}
    \usepackage{enumitem}
    \usepackage{xfrac}
    \usepackage{multicol}
    \usepackage[textwidth=0.8in,textsize=tiny]{todonotes}
    \usepackage{listings}
    \usepackage{newfloat}
    \DeclareFloatingEnvironment[
        fileext=los,
        listname={List of Algorithms},
        name=Algorithm,
        placement=tbhp,
        within=section,
    ]{algorithm}
    
    \definecolor{Francesco}{HTML}{0066ff}
\definecolor{Riccardo}{HTML}{7600b5}
    
\usepackage[backref=true,giveninits=true,doi=false,url=false,isbn=false,backend=biber]{biblatex}
    \newbibmacro{string+doiurlisbn}[1]{%
      \iffieldundef{doi}{%
        \iffieldundef{url}{%
          \iffieldundef{isbn}{%
            \iffieldundef{issn}{%
              #1%
            }{%
              \href{http://books.google.com/books?vid=ISSN\thefield{issn}}{#1}%
            }%
          }{%
            \href{http://books.google.com/books?vid=ISBN\thefield{isbn}}{#1}%
          }%
        }{%
          \href{\thefield{url}}{#1}%
        }%
      }{%
        \href{http://dx.doi.org/\thefield{doi}}{#1}%
      }%
    }
    \DeclareFieldFormat{title}{\usebibmacro{string+doiurlisbn}{\mkbibemph{#1}}}
    \DeclareFieldFormat[article,incollection,thesis,misc,inproceedings,online,inbook]{title}{\usebibmacro{string+doiurlisbn}{\mkbibquote{#1}}}
    \DeclareFieldFormat{year}{\usebibmacro{year-parenthesis}{\mkbibemph{#1}}}
\renewbibmacro*{date}{%
\iffieldundef{year}
{}
{\ifentrytype{misc}{\printtext[parens]{\printdate}}{\printdate}}}%
\usepackage{hyperref}
    \hypersetup{colorlinks=true,allcolors=blue}
    \usepackage[capitalize,nameinlink,noabbrev]{cleveref}

\usepackage{amsthm}
    \theoremstyle{plain}
        \newtheorem{theorem}{Theorem}[section]
        \newtheorem{lemma}[theorem]{Lemma}
        \newtheorem{proposition}[theorem]{Proposition}
        \newtheorem{corollary}[theorem]{Corollary}

    \theoremstyle{definition}

    \theoremstyle{remark}

        \newtheorem{remark}[theorem]{Remark}

\input{shortcuts.tex}

\makeatletter
\@namedef{subjclassname@2020}{\textup{2020} Mathematics Subject Classification}
\makeatother

\title[How big is the image of the Galois representations attached to CM elliptic curves?]{How big is the image of the Galois representations \\ attached to CM elliptic curves?}
    \author{Francesco Campagna}
    \author{Riccardo Pengo}
    \address{Francesco Campagna - Max Planck Institute for Mathematics, Vivatsgasse 7, 53111 Bonn, Germany}
    \email{\href{mailto:campagna@mpim-bonn.mpg.de}{campagna@mpim-bonn.mpg.de}}
    \address{Riccardo Pengo - École normale supérieure de Lyon, Unit\'e de Math\'ematiques Pures et Appliqu\'ees, 46 all\'ee d'Italie, 69007 Lyon, France}
    \email{\href{mailto:riccardo.pengo@ens-lyon.fr}{riccardo.pengo@ens-lyon.fr}}
    \date{}

\begin{document}
    \maketitle
    \begin{abstract}
        Using an analogue of Serre's open image theorem for elliptic curves with complex multiplication, one can associate to each CM elliptic curve $E$ defined over a number field $F$ a natural number $\mathcal{I}(E/F)$ which describes how big the image of the Galois representation associated to $E$ is.
        We show how one can compute $\mathcal{I}(E/F)$, using a closed formula that we obtain from the classical theory of complex multiplication.
    \end{abstract}

%\vspace{3\baselineskip}

\section{Introduction}
\label{sec:introduction}

Fix an algebraic closure $\overline{\mathbb{Q}}$ of the field of rational numbers $\mathbb{Q}$. Let $E$ be an elliptic curve defined over a number field $F \subseteq \overline{\mathbb{Q}}$, and let:
\begin{equation} \label{eq:general_galois_representation}
    \rho_E: G_F \rightarrow \Aut_\mathbb{Z}(E_{\text{tors}})
\end{equation}
be the representation of the absolute Galois group $G_F := \operatorname{Gal}(\overline{F}/F)$ associated to its action on the torsion points $E_\text{tors} := E(\overline{F})_\text{tors}$ of the elliptic curve $E$.

If $E$ does not have complex multiplication (CM), \textit{i.e.} $\operatorname{End}_{\overline{F}}(E) \cong \mathbb{Z}$, Serre's open image theorem \cite[Théorème~3]{Serre_1971} implies that the index $\mathcal{I}(E/F) := \lvert \operatorname{Aut}_\mathbb{Z}(E_\text{tors}) \colon \rho_E(G_F) \rvert$ is finite.
One is naturally led to investigate the dependence of $\mathcal{I}(E/F)$ on $E$ and $F$. 
For instance, one can ask whether there exists an explicit, closed formula for $\mathcal{I}(E/F)$, whose terms can be effectively computed starting from a Weierstra{\ss} equation of $E$.
At the time of writing, and to the best of the authors' knowledge, no such formula is available in the literature.
The previous question can then be weakened, by asking whether there exists an upper bound for $\mathcal{I}(E/F)$, which can be effectively computed in terms of $E$.
An affirmative answer to this second question has been provided by Lombardo in \cite{Lombardo_2015}.
%In fact, it has even been conjectured, as explained for $F = \mathbb{Q}$ in the introduction to \cite{Rouse_Sutherland_Zureick-Brown_2021}, that there should exist such an upper bound which does not depend on $E$, but only on the field of definition $F$.
In fact, it has even been conjectured that there should exist such an upper bound which does not depend on $E$, but only on the field of definition $F$.
This conjecture is explicitly mentioned for $F = \mathbb{Q}$ in the introduction to the recent work of Rouse, Sutherland and Zureick-Brown \cite{Rouse_Sutherland_Zureick-Brown_2021}, and is known to hold true under the assumption of Serre's uniformity conjecture, by previous work of Zywina (see \cite[Theorem~1.4]{Zywina_2015}).

On the other hand, if $E$ has complex multiplication by an order $\mathcal{O}$ in an imaginary quadratic field $K$, \textit{i.e.} $\operatorname{End}_{\overline{F}}(E) \cong \mathcal{O}$, the index of the image of $\rho_E$ inside $\operatorname{Aut}_\mathbb{Z}(E_\text{tors})$ is infinite. 
Nevertheless, as we recall in \cref{sec:CM_open_image}, one can formulate an analogue of Serre's open image theorem for $E$, by replacing $\operatorname{Aut}_\mathbb{Z}(E_\text{tors})$ with a smaller subgroup $\mathcal{G}(E/F) \subseteq \operatorname{Aut}_\mathbb{Z}(E_\text{tors})$, explicitly defined in \eqref{eq:G(E/F)}, which is closed and of infinite index inside $\operatorname{Aut}_\mathbb{Z}(E_\text{tors})$.
As a consequence, the index $\mathcal{I}(E/F) := \lvert \mathcal{G}(E/F) \colon \rho_E(G_F) \rvert$ is finite, and, as above, one can ask whether it can be expressed by means of an explicit and closed formula.
The main goal of this paper is to show how to use the classical theory of complex multiplication to give the following affirmative answer to this question.

% More precisely, the following theorem, proved in \cref{sec:proof}, provides an explicit formula for the index $\mathcal{I}(E/F)$.

\begin{theorem} \label{thm:main_theorem}
Let $\mathcal{O}$ be an order in an imaginary quadratic field $K \subseteq \overline{\mathbb{Q}}$. Let $E$ be an elliptic curve that has complex multiplication by $\mathcal{O}$ and is defined over a number field $F \subseteq \overline{\mathbb{Q}}$.
Denote by $K^{\text{ab}} \subseteq \overline{\mathbb{Q}}$ the maximal abelian extension of $K$ contained in $\overline{\mathbb{Q}}$, and by  $F K \subseteq \overline{\mathbb{Q}}$ and $F K^\text{ab} \subseteq \overline{\mathbb{Q}}$ the composita of $F$ with $K$ and $K^\text{ab}$ respectively.
Then:
\begin{equation} \label{eq:CM_index_formula}
    \mathcal{I}(E/F) = [(F K) \cap K^{\text{ab}} \colon H_\mathcal{O}] \cdot \frac{\lvert \mathcal{O}^{\times} \rvert}{[F(E_{\text{tors}}): F K^{\text{ab}}] }
\end{equation}
where $H_\mathcal{O} \subseteq K^\text{ab}$ is the \emph{ring class field} of $K$ relative to the order $\mathcal{O}$ (see \cite[\S~9]{Cox_2013}), and $F(E_\text{tors}) \subseteq \overline{\mathbb{Q}}$ is the field obtained by adjoining to $F$ all the coordinates of all the points lying in $E_\text{tors}$.
\end{theorem}

Note that the right-hand side of \eqref{eq:CM_index_formula} makes sense because the extension $K \subseteq H_\mathcal{O}$ is abelian, and, whenever $\operatorname{End}_{\overline{F}}(E) \cong \mathcal{O}$, one knows that $F K^\text{ab} \subseteq F(E_\text{tors})$ \cite[\S~4.1 and Remark~3.8]{Campagna_Pengo_2020}, and $H_\mathcal{O} = K(j(E)) \subseteq F K$ \cite[Theorem~11.1]{Cox_2013}, where $j(E) \in F$ denotes the \textit{$j$-invariant} of the elliptic curve $E$. 
Moreover, the classical theory of complex multiplication implies that the degree of the extension $F K^\text{ab} \subseteq F(E_\text{tors})$ is finite and divides $\lvert \mathcal{O}^\times \rvert$.
We explain this in more detail in \cref{sec:proof}, which is mainly devoted to the proof of \cref{thm:main_theorem}.

As an immediate consequence of \cref{thm:main_theorem}, one has the divisibility:
\begin{equation} \label{eq:upper_bound}
    \left. \mathcal{I}(E/F) \, \middle| \, [ (F K) \cap K^\text{ab} \colon H_\mathcal{O} ] \cdot \lvert \mathcal{O}^\times \rvert \right.
\end{equation}
which shows that $\mathcal{I}(E/F)$ can be bounded solely in terms of $F$, for every CM elliptic curve $E_{/F}$. 
This improves the upper bounds for $\mathcal{I}(E/F)$ previously proved by Lombardo \cite[Theorem~6.6]{Lombardo_2017} and Bourdon and Clark \cite[Corollary~1.5]{Bourdon_Clark_2020}.
Moreover, \cref{thm:main_theorem} applied to any elliptic curve $E_{/\mathbb{Q}}$ which has complex multiplication by an imaginary quadratic order $\mathcal{O}$ shows that $\mathcal{I}(E/\mathbb{Q}) = \lvert \mathcal{O}^\times \rvert$. In the case $\mathcal{O} = \mathbb{Z}[i]$, this strengthens the conclusion of \cite[Theorem~1.3]{Lozano-Robledo_2019}. 

The foregoing discussion shows that $\mathcal{I}(E/F)$ is very well understood in the CM case. However, it may not appear immediately clear how to apply \eqref{eq:CM_index_formula} to compute $\mathcal{I}(E/F)$ in concrete examples.
We explain how to do so in \cref{sec:algorithm}. 
In fact, after rewriting \eqref{eq:CM_index_formula} appropriately (see \cref{prop:finite_computation}), we obtain an algorithm that takes as inputs a number field $F$ and a CM elliptic curve $E_{/F}$, and outputs $\mathcal{I}(E/F)$.
More precisely, we rephrase \cref{eq:CM_index_formula} in terms of a finite extension $L \supseteq F K$ such that $F(E_\text{tors}) = L K^\text{ab}$. 
We prove in \cref{prop:3_division_field} that one can always take $L = F(E[I])$ to be the \textit{$I$-division field} generated by the coordinates of the points $P \in E[I]$ belonging to the \textit{$I$-torsion subgroup}:  
\[
    E[I] := \bigcap_{\alpha \in I} \ker\left(E(\overline{F}) \xrightarrow{ [\alpha]_E } E(\overline{F})\right)
\]
where $I \subseteq \mathcal{O}$ is any ideal such that $\lvert \mathbb{Z}/(I \cap \mathbb{Z}) \rvert > \max(2,\lvert \mathcal{O}^\times \rvert/2)$, and $[\cdot]_E \colon \mathcal{O} \altxrightarrow{\sim} \operatorname{End}_{\overline{F}}(E)$ is the normalised isomorphism described in \cref{lem:normalized_isomorphism}.
In practice, if $j(E) \neq 0$ one usually takes $L = F(E[3])$ in order to ease the computational burden.
We devote \cref{sec:examples} to the application of this algorithm to some explicit examples of elliptic curves $E$ that have complex multiplication by imaginary quadratic orders $\mathcal{O}$ of class number two.

\section{Analogues of Serre's open image theorem for CM elliptic curves} \label{sec:CM_open_image}

Let $E$ be an elliptic curve defined over a number field $F \subseteq \overline{\mathbb{Q}}$. 
Then, the absolute Galois group $G_F$ naturally acts both on the set $E_\text{tors} = \varinjlim_N E[N]$, and on the adelic Tate module $\mathcal{T}(E) := \varprojlim_N E[N]$. 
The first action gives rise to the Galois representation $\rho_E$ appearing in \eqref{eq:general_galois_representation}, whereas the action on $\mathcal{T}(E)$ induces another Galois representation $\varrho_E \colon G_F \to \operatorname{Aut}_{\widehat{\mathbb{Z}}}(\mathcal{T}(E))$. 
As done in \cite[\S~4.1, Remarque~(1)]{Serre_1971}, one can construct an isomorphism: \[\nu \colon \operatorname{Aut}_{\widehat{\mathbb{Z}}}(\mathcal{T}(E)) \altxrightarrow{\sim} \operatorname{Aut}_{\widehat{\mathbb{Z}}}(E_\text{tors}) = \operatorname{Aut}_\mathbb{Z}(E_\text{tors})\] such that $\rho_E = \nu \circ \varrho_E$.
As a consequence, one can indifferently study the Galois representation $\rho_E$, as done in this paper, or its twin $\varrho_E$, as done in some of our references. 

If $E$ does not have complex multiplication, \textit{i.e.} if $\operatorname{End}_F(E) \cong \operatorname{End}_{\overline{F}}(E) \cong \mathbb{Z}$, then the celebrated ``open image theorem'', proved by Serre in \cite[Théorème~3]{Serre_1971}, shows that the image of the Galois representation $\rho_E$ is a subgroup of finite index inside $ \operatorname{Aut}_\mathbb{Z}(E_\text{tors}) \cong \operatorname{GL}_2(\widehat{\mathbb{Z}})$, where $\widehat{\mathbb{Z}} := \varprojlim_N (\mathbb{Z}/N \mathbb{Z})$ denotes the profinite completion of $\mathbb{Z}$.
On the other hand, if the elliptic curve $E$ has complex multiplication, the image of $\rho_E$ is not open inside $\operatorname{Aut}_\mathbb{Z}(E_\text{tors})$. 
However, one can formulate a CM analogue of Serre's open image theorem by replacing $\operatorname{Aut}_\mathbb{Z}(E_\text{tors})$ with an appropriate closed subgroup $\mathcal{G}(E/F) \subseteq \Aut_\mathbb{Z}(E_\text{tors})$, which we now describe.

Suppose now that $\operatorname{End}_{\overline{F}}(E) \not\cong \mathbb{Z}$.
Then the endomorphism ring $\operatorname{End}_{\overline{F}}(E)$ can be canonically identified with an order inside an imaginary quadratic field, as the following classical lemma shows.

\begin{lemma} \label{lem:normalized_isomorphism}
    Let $F$ be a number field, and $E_{/F}$ be an elliptic curve such that $\operatorname{End}_{\overline{F}}(E) \not\cong \mathbb{Z}$, where $\overline{F}$ denotes a fixed algebraic closure of $F$.
    Then, there exists an imaginary quadratic field $K$ and an order $\mathcal{O} \subseteq K$ such that $\operatorname{End}_{\overline{F}}(E) \cong \mathcal{O}$.
    Moreover, for each embedding $\iota \colon K \hookrightarrow \overline{F}$, there exists a unique isomorphism:
    \[            [\cdot]_{E,\iota} \colon \mathcal{O} \altxrightarrow{\sim} \operatorname{End}_{\overline{F}}(E)
    \]
    such that $[\alpha]_{E,\iota}^\ast(\omega) = \iota(\alpha) \cdot \omega$ for every $\alpha \in \mathcal{O}$ and every invariant differential $\omega$ defined over $E_{\overline{F}}$, where $[\alpha]_{E,\iota}^\ast(\omega)$ denotes the pull-back of $\omega$ along the endomorphism $[\alpha]_{E,\iota}$.
\end{lemma}
\begin{proof}
    See \cite[Chapter~III, Corollary~9.4]{si09} for the existence of $K$ and $\mathcal{O}$. Moreover, the existence of $[\cdot]_{E,\iota}$ follows from \cite[Chapter~II, Proposition~1.1]{si94}, after fixing an embedding $\overline{F} \hookrightarrow \mathbb{C}$.
    Finally, observe that for any two isomorphisms $[\cdot],[\cdot]' \colon \mathcal{O} \altxrightarrow{\sim} \operatorname{End}_{\overline{F}}(E)$, there exists an automorphism $\sigma \colon \mathcal{O} \altxrightarrow{\sim} \mathcal{O}$ such that $[\alpha] = [\sigma(\alpha)]'$ for every $\alpha \in \mathcal{O}$.
    Hence, if these isomorphisms satisfy the requirements of the lemma, we see that $\iota(\alpha - \sigma(\alpha)) \cdot \omega = 0$ for every $\alpha \in \mathcal{O}$ and every invariant differential $\omega$. Thus, we have that $\sigma = \operatorname{Id}_\mathcal{O}$, which allows us to conclude.
\end{proof}

Now, suppose at first that $E_{/F}$ is an elliptic curve such that  $\operatorname{End}_F(E) \cong \operatorname{End}_{\overline{F}}(E) \cong \mathcal{O}$ for some order $\mathcal{O}$ inside an imaginary quadratic field $K$. Then by \cite[Chapter~II, Proposition~30]{Shimura_1998} we necessarily have $K \subseteq F$ and one can easily show (using for instance \cite[Chapter~II, Theorem 2.2]{si94}) that the absolute Galois group $G_F$ of $F$ acts as $\mathcal{O}$-module automorphisms on $E_\text{tors}$. 
Thus, we see that: 
\begin{equation} \label{eq:G_Aut_O}
    \rho_E(G_F) \subseteq \Aut_\mathcal{O}(E_\text{tors}) =: \mathcal{G}(E/F)
\end{equation} 
where $\mathcal{G}(E/F)$ is an abelian group canonically isomorphic to $\widehat{\mathcal{O}}^\times$, the unit group of the profinite completion $\widehat{\mathcal{O}}:=\varprojlim_N (\mathcal{O}/N \mathcal{O})$. 
In particular, the extension $F \subseteq F(E_\text{tors})$ is abelian.
Note also that $\Aut_\mathcal{O}(E_\text{tors})$ is closed inside $\Aut_{\mathbb{Z}}(E_\text{tors})$, since we have:
\[
\Aut_\mathcal{O}(E_\text{tors})=\bigcap_{N \in \mathbb{N}} \text{res}_N^{-1}(\Aut_\mathcal{O}(E[N]))
\]
where $\text{res}_N:\Aut_{\mathbb{Z}}(E_\text{tors}) \to \Aut_{\mathbb{Z}}(E[N])$ denotes the natural restriction map. 
On the other hand, $\operatorname{Aut}_\mathcal{O}(E_\text{tors})$ is not open inside $\operatorname{Aut}_\mathbb{Z}(E_\text{tors}) \cong \operatorname{GL}_2(\widehat{\mathbb{Z}})$, because the latter does not contain any abelian subgroup of finite index.
However, the subgroup $\rho_E(G_F)$ is open in $\Aut_\mathcal{O}(E_\text{tors})$, as shown in \cite[\S~4.5]{Serre_1971} using the classical theorems of complex multiplication. 
Since $\Aut_\mathcal{O}(E_\text{tors}) \cong \widehat{\mathcal{O}}^\times$ is a profinite group, this in particular implies that the index of $\rho_E(G_F)$ inside $\Aut_\mathcal{O}(E_\text{tors})$ is finite. 
We can regard this result as an analogue of Serre's open image theorem for those CM elliptic curves whose field of definition contains the field $K$.

Assume now that the elliptic curve $E_{/F}$ satisfies $\operatorname{End}_F(E) \cong \mathbb{Z}$ and $\operatorname{End}_{\overline{F}}(E) \cong \mathcal{O}$, for some order $\mathcal{O}$ inside an imaginary quadratic field $K$. Again by \cite[Chapter~II, Proposition~30]{Shimura_1998}, under these assumptions we must have $K \not \subseteq F$. 
Since not all the geometric endomorphisms of $E$ are defined over the base field, in this case the Galois group $G_F$ does not respect the $\mathcal{O}$-module structure on $E_\text{tors}$. 
More precisely, since we fixed an embedding $\mathcal{O} \subseteq K \subseteq \overline{\mathbb{Q}} = \overline{F}$, there exists a unique isomorphism $[\cdot]_E: \mathcal{O} \altxrightarrow{\sim} \End_{\overline{F}}(E)$ such that for every $\alpha \in \mathcal{O}$ and every invariant differential $\omega$ on the elliptic curve $E_{\overline{F}}$, the equality $[\alpha]_E^\ast(\omega) = \alpha \omega$ holds.  then an automorphism $\sigma \in G_F$ acts on $[\alpha]_E (P)$ as:
\begin{equation} \label{eq:Galois_action_constants}
    \sigma \left([\alpha]_E (P) \right) = [\sigma (\alpha)]_E (\sigma(P))
\end{equation}
as follows from \cite[Chapter~II, Theorem 2.2]{si94}. 
We then see that for every $\sigma \in G_F$ and each fixed $\tau \in G_F$ restricting to the unique non-trivial element in $\Gal(FK/F)$, exactly one among $\sigma$ and $\sigma \tau$ acts $\mathcal{O}$-linearly on $E_\text{tors}$. 
We deduce that:
\begin{equation} \label{eq:fake_normalizer}
    \rho_E\left( G_F \right) \subseteq \left \langle \Aut_\mathcal{O}(E_\text{tors}), \rho_E(\tau) \right \rangle := \mathcal{G}(E/F)
\end{equation}
and one can easily show that the group $\mathcal{G}(E/F)$ does not actually depend on $\tau$, thus justifying the notation. 
Indeed, if both $\tau, \tau' \in G_F$ restrict to the unique non-trivial element of $\operatorname{Gal}(F K/F)$, one has that $\tau \tau' \in \operatorname{Gal}(\overline{F}/F K)$, hence $\rho_E(\tau \tau') \in \operatorname{Aut}_\mathcal{O}(E_\text{tors})$, which shows that $\left \langle \Aut_\mathcal{O}(E_\text{tors}), \rho_E(\tau) \right \rangle = \left \langle \Aut_\mathcal{O}(E_\text{tors}), \rho_E(\tau') \right \rangle$ as wanted.
Moreover, $\rho_E(\tau)$ normalises  $\operatorname{Aut}_\mathcal{O}(E_\text{tors})$, as follows from \eqref{eq:Galois_action_constants} and the fact that $\rho_E(\tau)^2 \in \Aut_\mathcal{O}(E_\text{tors})$. 
Hence, we see that $\Aut_\mathcal{O}(E_\text{tors})$ is a normal subgroup of $\mathcal{G}(E/F)$ with index $|\mathcal{G}(E/F):\Aut_\mathcal{O}(E_\text{tors})|=2$. 
As a consequence, $\mathcal{G}(E/F)$ is closed inside $\Aut_\mathbb{Z}(E_\text{tors})$, and so it is a profinite group. 
On the other hand, $\mathcal{G}(E/F)$ is not open inside $\Aut_\mathbb{Z}(E_\text{tors})$, because it contains the abelian group $\operatorname{Aut}_\mathcal{O}(E_\text{tors})$ as a finite-index subgroup.
Thus, $\rho_E(G_F)$ cannot be open inside $\operatorname{Aut}_\mathbb{Z}(E_\text{tors})$. Nevertheless, $\rho_E(G_F)$ is open inside the closed subgroup $\mathcal{G}(E/F)$, as the following lemma shows.

\begin{lemma} \label{lem: base change index}
    Let $E_{/F}$ be an elliptic curve with complex multiplication by an order $\mathcal{O}$ in an imaginary quadratic field $K \not \subseteq F$, and let $E':=E_{FK}$ denote the base-change of $E$ to the compositum $FK$. Then $\rho_E(G_F)$ is open in $\mathcal{G}(E/F)$, and the following equality:
    \[
    \mathcal{I}(E/F) := \lvert \mathcal{G}(E/F) : \rho_E(G_F)  \rvert =  \lvert \Aut_\mathcal{O}(E'_\text{tors}) : \rho_{E'}(G_{FK})  \rvert =: \mathcal{I}(E/F K)
    \]
    holds.
\end{lemma}
\begin{proof}
Since $\Aut_\mathcal{O}(E_\text{tors})$ is closed and of finite index in $\mathcal{G}(E/F)$, it is also open in the same group. 
Moreover, by \cite[\S~4.5, Corollaire]{Serre_1971} the inclusion $\rho_{E'}(G_{FK}) \subseteq \Aut_\mathcal{O}(E'_\text{tors})$ is open, and clearly $\rho_{E'}(G_{FK})=\rho_{E}(G_{FK})$ and $\Aut_\mathcal{O}(E'_\text{tors})=\Aut_\mathcal{O}(E_\text{tors})$. 
Thus we see that $\rho_E(G_{FK})$ is an open subgroup of $\rho_E(G_{F})$ and we conclude that the latter is open in $\mathcal{G}(E/F)$.
In particular, $\rho_E(G_F)$ is a closed subgroup of finite index inside $\mathcal{G}(E/F)$.

To prove the equality of indices, we use the fact that $FK \subseteq F(E_\text{tors})$, by \cite[Lemma 3.15]{Bourdon_Clark_Stankewicz_2017}. Since $\rho_E$ induces an injective Galois representation $\Gal(F(E_\text{tors})/F) \hookrightarrow \mathcal{G}(E/F)$, we have $\left| \rho_E(G_F) : \rho_E(G_{FK}) \right|=2$. 
Now, the computation:
\[
\left| \mathcal{G}(E/F): \rho_E(G_F) \right|=\frac{1}{2} \left| \mathcal{G}(E/F): \rho_E(G_{FK}) \right|=\left| \Aut_\mathcal{O}(E_\text{tors}): \rho_E(G_{FK}) \right|
\]
allows us to conclude.
\end{proof}

We summarise our discussion so far. Given a number field $F$ and an elliptic curve $E_{/F}$ with complex multiplication by an order $\mathcal{O}$ in an imaginary quadratic field $K$, we define, following \eqref{eq:G_Aut_O} and \eqref{eq:fake_normalizer}:
\begin{equation} \label{eq:G(E/F)}
    \mathcal{G}(E/F) := \begin{cases}
        \Aut_\mathcal{O}(E_\text{tors}) &\text{ if } K\subseteq F, \\
        \left\langle \Aut_\mathcal{O}(E_\text{tors}), \rho_E(\tau) \right\rangle &\text{ if } K \not \subseteq F
    \end{cases}
\end{equation}
where, if $K \not\subseteq F$, we let $\tau \in G_F$ be any automorphism that restricts to the unique non-trivial element of $\operatorname{Gal}(F K/F)$. 
Then, in the previous discussion, we have shown that $\mathcal{G}(E/F)$ is a profinite group, which contains $\rho_E(G_F)$ as an open subgroup.
Moreover, if we define the \textit{CM index} $\mathcal{I}(E/F)$ to be: 
\begin{equation} \label{eq:notation_index}
    \mathcal{I}(E/F):= \left| \mathcal{G}(E/F) : \rho_E(G_F) \right|
\end{equation}
then by \cref{lem: base change index} we have that $\mathcal{I}(E/F)=\mathcal{I}(E/FK)$ is finite. 

\section{A formula for the index}
\label{sec:proof}

The aim of this section is to provide a proof of \cref{thm:main_theorem}. We place ourselves in the setting of the theorem, by fixing an order $\mathcal{O}$ inside an imaginary quadratic field $K \subseteq \overline{\mathbb{Q}}$ and an elliptic curve $E$ which has complex multiplication by $\mathcal{O}$ and is defined over a number field $F \subseteq \overline{\mathbb{Q}}$. 
We explained in \cref{lem: base change index} that $\mathcal{I}(E/F) = \mathcal{I}(E/F K)$, hence we will assume without loss of generality that $K \subseteq F$. This in particular implies that $H_\mathcal{O} \subseteq F$, where, as in \cref{thm:main_theorem}, $H_\mathcal{O}$ denotes the ring class field of $K$ relative to the order $\mathcal{O}$.

The formula \eqref{eq:CM_index_formula} appearing in \cref{thm:main_theorem} is a byproduct of the first main theorem of complex multiplication (see \cite[Chapter 10, Theorem 8]{Lang_1987}). The latter asserts the existence of a unique continuous group homomorphism $\mu \colon \mathbb{A}_F^\times \to K^\times$ such that, for every $s \in \mathbb{A}_F^\times$ and every complex uniformisation $\xi \colon \mathbb{C} \twoheadrightarrow E(\mathbb{C})$ with $\Lambda := \ker(\xi) \subseteq K$, the following diagram:
    \[
        \begin{tikzcd}[column sep = 2.5cm]
            K/\Lambda \rar["\left( \mu(s) \operatorname{N}_{F/K}(s^{-1}) \right) \cdot"] \dar["\xi"] & K/\Lambda \dar["\xi"] \\
            E_\text{tors} \rar["{[s,F]}"] & E_\text{tors}
        \end{tikzcd}
    \]
    commutes. Here $\mathrm{N}_{F/K} \colon \mathbb{A}_F^\times \to \mathbb{A}_K^\times$ denotes the idelic norm map, $[\cdot,K] \colon \mathbb{A}_K^\times \twoheadrightarrow \operatorname{Gal}(K^\text{ab}/K)$ denotes the global Artin map, and the upper horizontal arrow is given by the idelic multiplication map (see \cite[Page~100]{Lang_1987}).
    In particular, the action of the idèle $\mu(s) \operatorname{N}_{F/K}(s^{-1}) \in \mathbb{A}_K^\times$ on the set of lattices contained in $K$, described in \cite[Chapter~8, Theorem~10]{Lang_1987}, fixes $\Lambda$. Since $\Lambda$ is an invertible fractional ideal of $\mathcal{O}$, this implies that $\mu(s) \operatorname{N}_{F/K}(s^{-1})$ fixes also $\mathcal{O}$.
    Thus, the finite idèle $(\mu(s) \operatorname{N}_{F/K}(s^{-1}))_\text{fin}$ lies in the subgroup $\widehat{\mathcal{O}}^\times \subseteq \mathbb{A}_K^\times$.
    Hence, the association $s \mapsto (\mu(s) \operatorname{N}_{F/K}(s^{-1}))_\text{fin}$ defines a continuous group homomorphism $\theta_E \colon \mathbb{A}_F^\times \to \widehat{\mathcal{O}}^\times$, which makes the following diagram:
    \begin{equation}
        \label{eq:main_theorem_CM_square}
        \begin{tikzcd}
        \mathbb{A}_F^{\times} \arrow[r,"\theta_E"] \arrow[d,two heads,"{\restr{[\cdot,F]}{F(E_\text{tors})}}" swap,outer sep=0.1cm] & \widehat{\mathcal{O}}^\times \arrow[d,"\sim" rotated_label] \\
        \operatorname{Gal}(F(E_{\text{tors}})/F) \arrow[r,hook,"\rho_E"] & \operatorname{Aut}_{\mathcal{O}}(E_{\text{tors}})
    \end{tikzcd}
    \end{equation}
    commute.
    We are now ready to prove \cref{thm:main_theorem}.

\begin{proof}[Proof of \cref{thm:main_theorem}]
    Define $\psi_E$ to be the surjective group homomorphism:
\[
        \begin{tikzcd}
            \psi_E \ \colon \ \operatorname{Aut}_{\mathcal{O}}(E_{\text{tors}})
            \cong \widehat{\mathcal{O}}^{\times} \arrow[r,two heads,"\mathfrak{a}_{\mathcal{O}}"] & \operatorname{Gal}(K^{\text{ab}}/H_{\mathcal{O}})
        \end{tikzcd}
\]
where $\mathfrak{a}_\mathcal{O} \colon \widehat{\mathcal{O}}^\times \twoheadrightarrow \operatorname{Gal}(K^\text{ab}/H_\mathcal{O})$ is the composition of the natural embedding $\widehat{\mathcal{O}}^\times \hookrightarrow \mathbb{A}_K^\times$ with the map $\mathbb{A}_K^\times \twoheadrightarrow G_K^\text{ab}$ given by $s \mapsto [s^{-1},K]$.
It is easy to show that $\psi_E$ fits in a short exact sequence:
\begin{equation} \label{eq:exact_sequence_psi}
        1 \to \operatorname{Aut}_F(E) \to \operatorname{Aut}_\mathcal{O}(E_\text{tors}) \xrightarrow{\psi_E} \operatorname{Gal}(K^\text{ab}/H_\mathcal{O}) \to 1
\end{equation}
because $\ker(\mathfrak{a}_\mathcal{O}) = \ker([\cdot,K]) \cap \widehat{\mathcal{O}}^\times = K^\times \cap \widehat{\mathcal{O}}^\times = \mathcal{O}^\times$.
Then, we can form the following square:
\begin{equation} \label{eq:surjections_square}
        \begin{tikzcd}
            \operatorname{Gal}(F(E_\text{tors})/F) \rar[hook,"\rho_E"] \dar[two heads] & \operatorname{Aut}_\mathcal{O}(E_\text{tors}) \dar[two heads,"\psi_E"] \\
            \operatorname{Gal}(K^\text{ab}/F \cap K^\text{ab}) \rar[hook,"\iota"] & \operatorname{Gal}(K^\text{ab}/H_\mathcal{O})
        \end{tikzcd}
\end{equation}
where the map on the left is defined by the composition:
    \[
        \operatorname{Gal}(F(E_\text{tors})/F) \twoheadrightarrow \operatorname{Gal}(F K^\text{ab}/F) \altxrightarrow{\sim} \operatorname{Gal}(K^\text{ab}/F \cap K^\text{ab})
    \]
of a restriction map and a natural isomorphism coming from Galois theory.
We claim that \eqref{eq:surjections_square} commutes.
Indeed, extending \eqref{eq:surjections_square} by diagram \eqref{eq:main_theorem_CM_square} gives the following square:
\begin{equation} \label{eq:functoriality_of_CFT}
        \begin{tikzcd}
        \mathbb{A}_F^{\times} \arrow[r,"\theta_E"] \arrow[d,two heads,"{\restr{[\cdot,F]}{K^{\text{ab}}}}" swap,outer sep=0.1cm] & \widehat{\mathcal{O}}^\times \arrow[d,"\mathfrak{a}_{\mathcal{O}}",two heads] \\
        \operatorname{Gal}(K^{\text{ab}}/F \cap K^{\text{ab}}) \arrow[r,hook,"\iota"] & \operatorname{Gal}(K^{\text{ab}}/H_{\mathcal{O}})
    \end{tikzcd}
\end{equation}
which commutes because, for every $s \in \mathbb{A}_F^\times$, one has: 
\[
    \mathfrak{a}_\mathcal{O}(\theta_E(s)) = [(\mu(s) \cdot \operatorname{N}_{F/K}(s^{-1}))^{-1},K] = [\operatorname{N}_{F/K}(s),K] = \iota(\restr{[s,F]}{K^\text{ab}})
\]
using the fact that $K^\times \cdot (K \otimes_\mathbb{Q} \mathbb{R})^\times \subseteq \ker([\cdot,K])$, as explained in \cite[Chapter~IX, Theorem~3]{Artin_Tate_1968}, and the functoriality of class field theory \cite[Chapter~VI, Proposition~5.2]{Neukirch_1999}.
Thus \eqref{eq:surjections_square} commutes, because \eqref{eq:functoriality_of_CFT} does, and the vertical maps in the commutative diagram \eqref{eq:main_theorem_CM_square} are surjective.

Now, \eqref{eq:exact_sequence_psi} and \eqref{eq:surjections_square} induce the following commutative diagram:
    \begin{equation} \label{eq:snake_diagram}
        \begin{tikzcd}
            1 \rar & \operatorname{Gal}(F(E_\text{tors})/F K^\text{ab}) \rar \dar[hook,dashed,"\iota'"] & \operatorname{Gal}(F(E_\text{tors})/F) \rar \dar[hook,"\rho_E"] \arrow[dr, phantom, "\eqref{eq:surjections_square}"] & \operatorname{Gal}(K^\text{ab}/F \cap K^\text{ab}) \rar \dar[hook,"\iota"] & 1 \\
            1 \rar & \operatorname{Aut}_F(E) \rar & \operatorname{Aut}_\mathcal{O}(E_\text{tors}) \rar["\psi_E" swap] & \operatorname{Gal}(K^\text{ab}/H_\mathcal{O}) \rar & 1
        \end{tikzcd}
    \end{equation}
    whose rows are exact. This shows in particular that the degree of the extension $F K^\text{ab} \subseteq F(E_\text{tors})$ is finite and divides $\lvert \operatorname{Aut}_F(E) \rvert = \lvert \mathcal{O}^\times \rvert$. Finally, the snake lemma gives:
    \[
        \mathcal{I}(E/F) = \lvert \operatorname{coker}(\rho_E) \rvert = \lvert \operatorname{coker}(\iota) \rvert \cdot \lvert \operatorname{coker}(\iota') \rvert = [F \cap K^{\text{ab}} \colon H_\mathcal{O}] \cdot \frac{\lvert \mathcal{O}^{\times} \rvert}{[F(E_{\text{tors}}): F K^{\text{ab}}] }
    \]
    which allows us to conclude.
\end{proof}

An immediate consequence of \cref{thm:main_theorem} is the following improvement of the bounds provided by \cite[Theorem~6.6]{Lombardo_2017} and \cite[Corollary~1.5]{Bourdon_Clark_2020}.

\begin{corollary}
Let $\mathcal{O}$ be an order inside an imaginary quadratic field $K$.
For every number field $F \subseteq \overline{\mathbb{Q}}$, and every elliptic curve $E_{/F}$ with complex multiplication by $\mathcal{O}$, the index $\mathcal{I}(E/F)$ divides $[(F K) \cap K^\text{ab} \colon H_\mathcal{O}] \cdot \lvert \mathcal{O}^\times \rvert$.
\end{corollary}

Moreover, \cref{thm:main_theorem} can be rephrased in a simpler fashion, if one assumes that $\lvert \mathcal{O}^\times \rvert = 2$, which holds for every order $\mathcal{O}$ of discriminant $\Delta_\mathcal{O} < -4$.

\begin{corollary} \label{cor:index_non_0_1728}
Let $\mathcal{O}$ be an order inside an imaginary quadratic field $K$, and suppose that $\Delta_\mathcal{O} < -4$.
Let $E$ be an elliptic curve with complex multiplication by $\mathcal{O}$, defined over a number field $F \subseteq \overline{\mathbb{Q}}$. Then, the following equality:
\begin{equation} \label{eq:index_j_not_0_1728}
    \frac{\mathcal{I}(E/F)}{[(F K) \cap K^{\text{ab}}:H_\mathcal{O}] } =
    \begin{cases}
        2, \ & \text{if} \ F(E_\text{tors}) = F K^\text{ab} \\
        1, \ & \text{otherwise} \\
    \end{cases}
\end{equation}
holds.
\end{corollary}

The dichotomy provided by \eqref{eq:index_j_not_0_1728} reflects a property of CM elliptic curves introduced by Shimura in \cite[Pages~216-218]{Shimura_1994}, and studied in \cite[\S~5]{Campagna_Pengo_2020}. 
In particular, \cref{cor:index_non_0_1728} generalises \cite[Corollary~5.8]{Campagna_Pengo_2020}, which was proved by different means.

\begin{remark}
    Specializing \cref{thm:main_theorem} to $F = \mathbb{Q}(j(E))$ we see that $\mathcal{I}(E/F) \in \{ 1, \lvert \mathcal{O}^\times \rvert \}$. 
    However, this does not allow to describe explicitly the image $\rho_E(G_F)$ as a subgroup of $\operatorname{Aut}_\mathbb{Z}(E_\text{tors}) \cong \operatorname{GL}_2(\widehat{\mathbb{Z}})$, since the latter can vary amongst infinitely many possible subgroups, as it happens already for $F = \mathbb{Q}$ (see \cite[Theorem~6.3]{Campagna_Pengo_2020}).
    On the other hand, the image of $\rho_E(G_F)$ under the natural projections $\operatorname{GL}_2(\widehat{\mathbb{Z}}) \twoheadrightarrow \operatorname{GL}_2(\mathbb{Z}_\ell)$ for $\ell \in \mathbb{N}$ a prime, belongs, up to conjugation, to a finite list of subgroups which has been explicitly determined by Lozano-Robledo \cite{Lozano-Robledo_2019}. 
\end{remark}

To conclude this section, we observe that \cref{thm:main_theorem} implies that the index $\mathcal{I}(E/F)$ is invariant under appropriate twisting of the elliptic curve $E$, as specified by the following corollary.

\begin{corollary}
\label{cor:index_abelian_twist}
Let $\mathcal{O}$ be an order inside an imaginary quadratic field $K$, and set $d := \lvert \mathcal{O}^\times \rvert$. 
Let $E_{/F}$ be an elliptic curve defined over a number field $F \subseteq \overline{\mathbb{Q}}$ such that $\operatorname{End}_F(E) \cong \mathcal{O}$.  
Suppose that $E$ is the twist of another elliptic curve $E'_{/F}$ by $\sqrt[d]{\alpha}$, for some $\alpha \in F^\times$ such that $L := F(\sqrt[d]{\alpha}) \subseteq F K^\text{ab}$.
Then $\mathcal{I}(E/F)=\mathcal{I}(E'/F)$.
\end{corollary}
\begin{proof}
First of all, note that the extension $F \subseteq L$ is well defined, because we have the inclusion $K \subseteq F$, by the hypothesis $\operatorname{End}_F(E) \cong \mathcal{O}$. Thus, the group of $d$-th roots of unity $\mathcal{O}^\times$ is also contained in $F$.
Then, one has: 
\begin{equation} \label{eq:twisting_formula}
    \rho_E(\sigma) = \rho_{E'}(\sigma) \cdot \chi_\alpha(\sigma)    
\end{equation}
for every $\sigma \in G_F$, where $\rho_{E} \colon G_F \to \mathcal{G}(E/F) \cong  \widehat{\mathcal{O}}^\times$ and $\rho_{E'} \colon G_F \to \mathcal{G}(E'/F) \cong  \widehat{\mathcal{O}}^\times$ are the Galois representations associated to $E$ and $E'$.
Moreover, $\chi_\alpha \colon G_F \to \mathcal{O}^\times \subseteq \widehat{\mathcal{O}}^\times$ is the Kummer character attached to the extension $F \subseteq L$, defined by the equality $\sigma(\sqrt[d]{\alpha}) = \chi_\alpha(\sigma) \cdot \sqrt[d]{\alpha}$ for every $\sigma \in G_F$.

Now, for every $\sigma \in \operatorname{Gal}(\overline{\mathbb{Q}}/L F(E'_\text{tors}))$, we have that $\rho_{E'}(\sigma) = \chi_\alpha(\sigma) = 1$, hence \eqref{eq:twisting_formula} implies that $\rho_E(\sigma) = 1$. Thus, the inclusion $F(E_\text{tors}) \subseteq L F(E'_\text{tors})$ holds.
On the other hand, if $\tau \in \operatorname{Gal}(\overline{\mathbb{Q}}/F(E_\text{tors}))$, the hypothesis $L \subseteq F K^\text{ab}$ and the inclusion $F K^\text{ab} \subseteq F(E_\text{tors})$ imply that $\tau$ fixes $L$, and thus that $\rho_E(\tau) = \chi_\alpha(\tau) = 1$.
Therefore, \eqref{eq:twisting_formula} gives that $\rho_{E'}(\tau) = 1$. Hence, the opposite inclusion $L F(E'_\text{tors}) \subseteq F(E_\text{tors})$ holds.
Thus, we have that $F(E_\text{tors}) = L F(E'_\text{tors}) = F(E'_\text{tors})$, where the last equality follows from the hypothesis $L \subseteq F K^\text{ab}$ and the inclusion $F K^\text{ab} \subseteq F(E'_\text{tors})$.
Finally, using \cref{thm:main_theorem}, one gets that $\mathcal{I}(E/F)=\mathcal{I}(E'/F)$, as we wanted to prove. 
\end{proof}

\section{How to compute the index in practice} \label{sec:algorithm}

In this section we show how one can concretely compute the index $\mathcal{I}(E/F)$ for any given CM elliptic curve $E$ defined over a number field $F$. 
Thanks to \cref{lem: base change index}, we can and will assume throughout this section, without loss of generality, that the number field $F$ contains the CM field $K$.

The starting point of our discussion is the formula \eqref{eq:CM_index_formula} provided by \cref{thm:main_theorem}.
Let us observe that \eqref{eq:CM_index_formula}, albeit completely explicit, involves the degree of the finite extension $F K^\text{ab} \subseteq F(E_\text{tors})$ which \textit{a priori} can not be implemented in a computer, because $F K^\text{ab}$ is an infinite algebraic extension of $\mathbb{Q}$. 
Nevertheless, the following result shows how one can rewrite \eqref{eq:CM_index_formula} as an equality involving only finite abelian groups and number fields.

\begin{proposition} \label{prop:finite_computation}
    Let $\mathcal{O}$ be an order inside an imaginary quadratic field $K \subseteq \overline{\mathbb{Q}}$. Fix a number field $F \subseteq \overline{\mathbb{Q}}$ and an elliptic curve $E_{/F}$ such that $\operatorname{End}_F(E) \cong \mathcal{O}$.
    Then, we have:
    \begin{equation} \label{eq:finite_computation}
        \mathcal{I}(E/ F) = \frac{\lvert \mathcal{O}^\times \rvert \cdot [L \cap K^\text{ab} \colon K]}{\lvert \operatorname{Pic}(\mathcal{O}) \rvert \cdot [L \colon F]}
    \end{equation}
    for every finite extension $F \subseteq L$ such that $F(E_\text{tors}) = L K^\text{ab}$ is the compositum of $L$ and $K^\text{ab}$ inside $\overline{\mathbb{Q}}$.
\end{proposition}
\begin{proof}
    Combining \cref{thm:main_theorem} with the equality:
    \[
        [F(E_\text{tors}) \colon F K^\text{ab}] = [L K^\text{ab} \colon F K^\text{ab}] = \frac{[L \colon F]}{[L \cap K^\text{ab} \colon F \cap K^\text{ab}]} = \frac{[L \colon F] [F \cap K^\text{ab} \colon K]}{[L \cap K^\text{ab} \colon K]}
    \]
    allows us to conclude, because $[F \cap K^\text{ab} \colon K] = [F \cap K^\text{ab} \colon H_\mathcal{O}] \cdot \lvert \operatorname{Pic}(\mathcal{O}) \rvert$.
\end{proof}

Using \cref{prop:finite_computation}, we can now reduce the computation of $\mathcal{I}(E/F)$ to the following steps:
\begin{enumerate}[label*=\protect\fbox{S.\arabic{enumi}}]
    \item compute $\lvert \mathcal{O}^\times \rvert$ and $\lvert \operatorname{Pic}(\mathcal{O}) \rvert$; \label{step:1}
    \item find a finite extension $F \subseteq L$ such that $F(E_\text{tors}) = L K^\text{ab}$, and compute $[L \colon F]$; \label{step:2}
    \item compute $[L \cap K^\text{ab} \colon K]$, \textit{i.e.} the degree of the maximal abelian sub-extension of $K \subseteq L$. \label{step:3}
\end{enumerate}
To achieve \labelcref{step:1} one can use for instance the algorithms described in \cite[\S~5.3]{Cohen_1993} for the computation of $\lvert \operatorname{Pic}(\mathcal{O}) \rvert$, and the fact that $\lvert \mathcal{O}^\times \rvert = 2$ unless $\mathcal{O} = \mathbb{Z}[i]$, for which $\lvert \mathcal{O}^\times \rvert = 4$, or $\mathcal{O} = \mathbb{Z}\left[\frac{1 + \sqrt{-3}}{2}\right]$, for which $\lvert \mathcal{O}^\times \rvert = 6$.
Moreover, once \labelcref{step:2} has been carried out, and the extension $F \subseteq L$ is known, one can deal with the last step \labelcref{step:3} in (at least) two different ways: 
\begin{itemize}
    \item one can use the isomorphism:
    \begin{equation} \label{eq:abelianized_S_n}
        \operatorname{Gal}(L \cap K^\text{ab}/K) \cong \operatorname{Gal}(L'/K)^\text{ab}
    \end{equation}
    where $K \subseteq L' \subseteq L$ denotes the maximal sub-extension of $K \subseteq L$ which is Galois over $K$ and the notation $S^\text{ab}$ stands for the abelianization of a finite group $S$ (\textit{i.e.} its maximal abelian quotient). In order to compute the right hand side of \eqref{eq:abelianized_S_n}, note that, if $G := \operatorname{Gal}(\widetilde{L}/K)$ denotes the Galois group of the Galois closure $\widetilde{L}$ of the extension $K \subseteq L$, and $H^G \subseteq G$ denotes the normal closure of the subgroup $H := \operatorname{Gal}(\widetilde{L}/L)$ inside $G$, then we have $\operatorname{Gal}(L'/K) \cong G/H^G$. Since both $G$ and $H$ can be computed as subgroups of the symmetric group $\mathfrak{S}_n$ on $n = [L \colon K]$ letters (see \cite[\S~6.3]{Cohen_1993}), the abelian group $(G/H^G)^\text{ab}$ can also be explicitly computed, for instance using the functions \href{https://www.gap-system.org/Manuals/doc/ref/chap39.html#X7BDEA0A98720D1BB}{\textsc{NormalClosure}} and \href{https://www.gap-system.org/Manuals/doc/ref/chap39.html#X7BB93B9778C5A0B2}{\textsc{MaximalAbelianQuotient}} in GAP \cite{GAP4};
    \item one can compute $[L \cap K^\text{ab} \colon K]$ as the index of the norm group $\mathrm{T}_{\mathfrak{m}}(L/K) \subseteq \operatorname{Cl}_\mathfrak{m}(K)$, where $\mathfrak{m} := \delta_{L/K}$ is the relative discriminant of $K \subseteq L$, and $\operatorname{Cl}_\mathfrak{m}(K)$ denotes the ray class group of $K$ modulo $\mathfrak{m}$ (see \cite[Chapter~VI, \S~7]{Neukirch_1999}).
    This norm group $\mathrm{T}_{\mathfrak{m}}(L/K)$ can be computed using an adaptation of \cite[Algorithm~4.4.5]{Cohen_2000} to the non-Galois case.
    More precisely:
    \begin{itemize}
        \item in the fourth step of the aforementioned algorithm, one can proceed even if the polynomials $T_j$ do not have the same degree, by taking as $f$ the greatest common divisor of their degrees.
    Indeed, $\mathrm{T}_{\mathfrak{m}}(L/K)$ is by definition generated by the classes of $\mathfrak{p}^{f(\mathfrak{P}/\mathfrak{p})}$, where $\mathfrak{p} := \mathfrak{P} \cap \mathcal{O}_K$ and $\mathfrak{P}$ varies amongst the prime ideals of $\mathcal{O}_L$ coprime with $\mathfrak{m} \cdot \mathcal{O}_L$, and the inertia degrees $f(\mathfrak{P}/\mathfrak{p})$ correspond exactly to the degrees of the polynomials $T_j$ mentioned above;
    \item in the second step of the same algorithm, one should always output the matrix $M$ even if $\det(M) \neq [L \colon K]$.
    In fact, $\det(M)$ will be precisely the index of the norm group inside $\operatorname{Cl}_\mathfrak{m}(K)$, \textit{i.e.} the equality $[L \cap K^\text{ab} \colon K] = \det(M)$ holds.
    \end{itemize}
    Note that this modification does indeed work (assuming the validity of the Generalised Riemann Hypothesis), because $\mathrm{T}_\mathfrak{m}(L/K) = \mathrm{T}_\mathfrak{m}(L \cap K^\text{ab}/K)$ by \cite[Chapter~XIV, Theorem~7]{Artin_Tate_1968}.
\end{itemize}

Thus, in order to have a complete procedure for the computation of the CM index $\mathcal{I}(E/F)$, we only need to prove that one can always find a finite extension $F \subseteq L$ such that $F(E_\text{tors}) = L K^\text{ab}$ as in \labelcref{step:2}. The next proposition shows that one can take $L$ to be essentially any division field.

\begin{proposition} \label{prop:3_division_field}
    Let $\mathcal{O}$ be an order inside an imaginary quadratic field $K$ and let $E_{/F}$ be an elliptic curve defined over a number field $F \subseteq \overline{\mathbb{Q}}$ such that $\operatorname{End}_F(E) \cong \mathcal{O}$.
    Fix an ideal $I \subseteq \mathcal{O}$ and let $L := F(E[I])$ be the $I$-division field associated to $E$. Then $F(E_\text{tors}) = L K^\text{ab}$ whenever $\lvert \mathbb{Z}/(I \cap \mathbb{Z}) \rvert > 2$ if $j(E) \neq 0$, and $\lvert \mathbb{Z}/(I\cap \mathbb{Z}) \rvert > 3$ otherwise.
\end{proposition}
\begin{proof}
    The inclusion $L K^\text{ab} \subseteq F(E_\text{tors})$ is clear, and the other containment can be proved as in \cite[Proposition~5.7]{Campagna_Pengo_2020}.
    More precisely, fix an embedding $\overline{\mathbb{Q}} \hookrightarrow \mathbb{C}$ and a complex uniformisation $\xi \colon \mathbb{C} \twoheadrightarrow E(\mathbb{C})$, such that $\ker(\xi) = \Lambda$ for some lattice $\Lambda \subseteq K$. Then \cite[Theorem~5.4]{Shimura_1994} shows that, for every field automorphism $\sigma \colon \mathbb{C} \to \mathbb{C}$ which fixes $F K^\text{ab}$, there exists a complex uniformisation $\xi' \colon \mathbb{C} \twoheadrightarrow E(\mathbb{C})$ such that $\sigma(\xi(z)) = \xi'(z)$ for every $z \in K$.
    This implies in particular that there exists $\varepsilon \in \mathcal{O}^\times$ such that $\sigma(P) = [\varepsilon]_E(P)$ for every $P \in E_\text{tors}$.
    If now $\sigma$ fixes also the division field $L = F(E[I])$, one must have $\varepsilon = 1$ by our assumptions on $I$. We conclude that $\sigma$ fixes the entire $F(E_\text{tors})$, which in turn implies that $F(E_\text{tors}) \subseteq L K^\text{ab}$ as we wanted to show. 
\end{proof}

Using \cref{prop:3_division_field}, we see that \labelcref{step:1}, \labelcref{step:2} and \labelcref{step:3} indeed describe a procedure to compute the index $\mathcal{I}(E/F)$ for any CM elliptic curve defined over any number field $F$. In practice, in \labelcref{step:2} it is convenient to choose a ``small'' division field $L=F(E[I])$, for instance by using $I=3\mathcal{O}$ (when $j(E) \neq 0$), which gives $[L \colon F] \leq 8$. 
However, if one already knows an elliptic curve $E'_{/F}$ such that $j(E')=j(E)$ and $F(E'_\text{tors}) = F K^\text{ab}$, then the subsequent \cref{prop:infinite_twist}, whose proof is analogous to that of \cref{cor:index_abelian_twist}, shows that one can take $L$ to be a Kummer extension of $F$ with degree $[L:F] \leq \lvert \mathcal{O}^\times \rvert \leq 6$. Since computations involving division fields of elliptic curves are typically hard, taking such an $L$ is certainly more advantageous in this situation.

\begin{proposition} \label{prop:infinite_twist}
    Let $\mathcal{O}$ be an order inside an imaginary quadratic field $K$, and set $d := \lvert \mathcal{O}^\times \rvert$. 
    Let $E_{/F}$ be an elliptic curve defined over a number field $F \subseteq \overline{\mathbb{Q}}$ such that $\operatorname{End}_F(E) \cong \mathcal{O}$.
    Suppose that there exists another elliptic curve $E'_{/F}$ such that $F(E'_\text{tors}) = F K^\text{ab}$, and that
    $E$ is the twist of $E'$ by $\sqrt[d]{\alpha}$, for some $\alpha \in F^\times$.
    Then $F(E_\text{tors}) = L K^\text{ab}$, where $L = F(\sqrt[d]{\alpha})$.
\end{proposition}
\begin{proof}
    If $\sigma \in \operatorname{Gal}(\overline{\mathbb{Q}}/L F(E'_\text{tors}))$, we see from the twisting formula \eqref{eq:twisting_formula} that $\rho_{E'}(\sigma) = \chi_\alpha(\sigma) = \rho_E(\sigma) = 1$, hence $F(E_\text{tors}) \subseteq L F(E'_\text{tors})$.
    Vice versa, if $\tau \in \operatorname{Gal}(\overline{\mathbb{Q}}/F(E_\text{tors}))$ then $\rho_E(\tau) = 1$ and $\rho_{E'}(\tau) = \chi_\alpha(\tau^{-1}) \in \mathcal{O}^\times$.
    However, \eqref{eq:snake_diagram} shows that $\rho_{E'}(G_F) \cap \mathcal{O}^\times = \{1\}$, because $F(E'_\text{tors}) = F K^\text{ab}$ by assumption. Hence $\rho_{E'}(\tau) = \chi_\alpha(\tau) = 1$, which allows us to conclude that $F(E_\text{tors}) = L F(E'_\text{tors}) = L K^\text{ab}$, as we wanted to show.
\end{proof}

\begin{remark}
    Note that the condition $F(E'_\text{tors}) = F K^\text{ab}$ is invariant under base change along a finite extension $F \subseteq F'$.
    In particular,
    if $\operatorname{Pic}(\mathcal{O}) = \{1\}$, one can take as $E'$ any base change to $F$ of an elliptic curve $E_{/K}$ which has complex multiplication by $\mathcal{O}$.
    On the other hand, if $\operatorname{Pic}(\mathcal{O}) \neq \{1\}$, constructing such an elliptic curve is a non-trivial matter, as we will see in the next section.
\end{remark}

\section{Explicit examples}
\label{sec:examples}

We now want to provide some examples of index computations for CM elliptic curves $E$ defined over the corresponding field of moduli $\mathbb{Q}(j(E))$. 
A way of constructing such curves is to consider an elliptic curve $\mathcal{E}$ defined over the function field $\mathbb{Q}(j)$, with $j$-invariant $j(\mathcal{E}) = j$ and discriminant $\Delta_\mathcal{E} \in \mathbb{Q}(j)$, and then specialise the parameter to $j = j_0$ for some CM $j$-invariant $j_0 \in \overline{\mathbb{Q}}$ such that $\Delta_\mathcal{E}(j_0) \neq 0$. When we want to emphasize that the specialization at $j_0$ of the elliptic curve $\mathcal{E}$ has complex multiplication by some order $\mathcal{O}$, we say that $j_0 \in \overline{\mathbb{Q}}$ is \textit{relative to the order} $\mathcal{O}$.
With a view towards doing explicit calculations in the mostly popular computer algebra systems in computational number theory, we consider and compare the following choices of $\mathcal{E}$:
\begin{enumerate}
    \item the curve:
    \[
    \mathcal{E}_\text{SAGE}: y^2 = x^3 + (-3j^2+5184j) x -2j^3+6912j^2-5971968j
    \]
    implemented in SageMath \cite{sagemath} under the command \href{https://doc.sagemath.org/html/en/reference/arithmetic_curves/sage/schemes/elliptic_curves/constructor.html?highlight=ellipticcurve_from_j#sage.schemes.elliptic_curves.constructor.EllipticCurve_from_j}{\textsc{EllipticCurve\_from\_j(j,False)}}. 
    We warn the reader that, without setting the second optional parameter equal to \textsc{False},
    the command \textsc{EllipticCurve\_from\_j}, applied to a rational number $j_0 \in \mathbb{Q}$, returns an elliptic curve $E_{/\mathbb{Q}}$ which has $j$-invariant $j(E) = j_0$, and minimal conductor among all its twists. This curve, in general, can be different from the specialization of $\mathcal{E}_\text{SAGE}$ at $j = j_0$;
    \item the curve:
    \[
        \mathcal{E}_\text{PARI}: y^2= x^3+ (-3j^2 + 5184j)x+ 2j^3 - 6912j^2 + 5971968j
    \]
    implemented in PARI/GP \cite{PARI/GP} under the command \href{https://pari.math.u-bordeaux.fr/dochtml/html-stable/Elliptic_curves.html#ellfromj}{\textsc{ellfromj(j)}};
    \item the curve:
    \begin{equation*}
        \mathcal{E}_\text{MAGMA}: y^2+xy=x^3-\frac{36}{j-1728}x - \frac{1}{j-1728}
    \end{equation*}
    implemented in MAGMA \cite{magma} under the command \href{http://magma.maths.usyd.edu.au/magma/handbook/text/1488#17001}{\textsc{EllipticCurveFromjInvariant(j)}}.
\end{enumerate}
The above families are clearly all defined over $\mathbb{Q}(j)$, and their singular specializations occur only at the values $j_0 \in \{0,1728\}$. Moreover, it is easily verified that  $\mathcal{E}_\text{PARI}$ and $\mathcal{E}_\text{SAGE}$ are isomorphic over $\mathbb{Q}(j, \sqrt{-1})$ while $\mathcal{E}_\text{SAGE}$ and $\mathcal{E}_\text{MAGMA}$ are isomorphic over $\mathbb{Q}\left(j, \sqrt{\frac{1728-j}{3}}\right)$.

Now, for every CM $j$-invariant $j_0 \in \overline{\mathbb{Q}}$ relative to an order of class number $2$, we want to compute the index $\mathcal{I}(E_{j_0}/\mathbb{Q}(j_0))$ where $E_{j_0}$ is the fiber over $j_0$ in any of the three families described above (one can check that all these fibers are non-singular). First of all, we show that for every CM invariant $j_0 \in \overline{\mathbb{Q}}$ the CM fibers $E_{j_0}$ in the above families have the same index $\mathcal{I}(E_{j_0}/\mathbb{Q}(j_0))$. 
% We need the following result.

% \begin{lemma} \label{lem:index_abelian_twist}
% Let $\mathcal{O}$ be an order in an imaginary quadratic field $K$ and let $E_{/H_\mathcal{O}}$ be an elliptic curve with complex multiplication by $\mathcal{O}$. Let $\alpha \in H_\mathcal{O}$ be such that $H_\mathcal{O}(\sqrt[d]{\alpha}) \subseteq K^\text{ab}$, where $d := \lvert \mathcal{O}^\times \rvert$, and denote by $E^{(\alpha)}$ the twist of $E$ by $\sqrt[d]{\alpha}$. Then $\mathcal{I}(E/H_\mathcal{O})=\mathcal{I}(E^{(\alpha)}/H_\mathcal{O})$.
% \end{lemma}
% \begin{proof}
%     Using the twisting formula $\rho_{E^{(\alpha)}} = \rho_E \cdot \chi_{\alpha}$, one easily sees that $H_\mathcal{O}(E^{(\alpha)}_\text{tors}) \subseteq H_\mathcal{O}(E_\text{tors})(\sqrt[d]{\alpha})$.
%     Moreover, the hypothesis $H_\mathcal{O}(\sqrt[d]{\alpha}) \subseteq K^\text{ab} \subseteq H_\mathcal{O}(E_\text{tors}) \cap H_\mathcal{O}(E_\text{tors}^{(\alpha)})$ implies that $H_\mathcal{O}(E_\text{tors})(\sqrt[d]{\alpha}) = H_\mathcal{O}(E_\text{tors})$ and $\chi_\alpha(\sigma) = 1$ for every automorphism $\sigma \in G_{H_\mathcal{O}}$ fixing $H_\mathcal{O}(E_\text{tors}^{(\alpha)})$.
%     This yields the equality $H_\mathcal{O}(E^{(\alpha)}_\text{tors}) = H_\mathcal{O}(E_\text{tors})(\sqrt[d]{\alpha}) = H_\mathcal{O}(E_\text{tors})$.
%     Finally, using \cref{thm:main_theorem}, one gets that $\mathcal{I}(E/H_\mathcal{O})=\mathcal{I}(E^{(\alpha)}/H_\mathcal{O})$, as we wanted to prove.
% \end{proof}

Fix now a CM $j$-invariant $j_0 \in \overline{\mathbb{Q}} \setminus \{0, 1728\}$ relative to an order $\mathcal{O}$. 
Let moreover $(E_{j_0},E_{j_0}',E_{j_0}'')$ be the specialisations of the families $(\mathcal{E}_\text{SAGE},\mathcal{E}_\text{PARI},\mathcal{E}_\text{MAGMA})$ to $j = j_0$. 
If $H_\mathcal{O} = K(j_0)$ denotes the ring class field relative to the order $\mathcal{O}$ then by \cref{lem: base change index} we have $\mathcal{I}(E_{j_0}/\mathbb{Q}(j_0))=\mathcal{I}(E_{j_0}/H_\mathcal{O})$ and similarly with the other two elliptic curves, so we assume that everything is base-changed to the ring class field. Since by the discussion above $E_{j_0}$ and $E'_{j_0}$ are twisted over $H_\mathcal{O}$ by $\alpha=-1$ and $H_{\mathcal{O}}(\sqrt{-1}) \subseteq K^\text{ab}$ (being the compositum of two abelian extensions of $K$), \cref{cor:index_abelian_twist} allows us to conclude that $\mathcal{I}(E_{j_0}/\mathbb{Q}(j_0))=\mathcal{I}(E'_{j_0}/\mathbb{Q}(j_0))$. 
Furthermore, the elliptic curve $\mathcal{E}_\text{MAGMA}$ admits a short Weierstra{\ss} form:
\[
    y^2 = x^{3} - \left(\frac{27 j}{j - 1728}\right) x + \frac{54 j}{j - 1728} 
\]
whose discriminant is given by $\Delta_j := 6^{12} \cdot j^2/(j - 1728)^3$. Thus, we see that: 
\[
    H_\mathcal{O}(\sqrt{j_0 - 1728}) = H_\mathcal{O}(\sqrt{\Delta_{j_0}}) \subseteq H_\mathcal{O}(E_{j_0}''[2])
\]
for every CM $j$-invariant $j_0 \in \overline{\mathbb{Q}}$, relative to the order $\mathcal{O}$.
Since $H_\mathcal{O}(E_{j_0}''[2])$ is generated over $H_\mathcal{O}$ by the Weber functions evaluated at $2$-torsion points, we have that $H_\mathcal{O}(E_{j_0}''[2]) \subseteq K^\text{ab}$ (see \cite[Theorem~4.7]{Campagna_Pengo_2020}). Thus $H_\mathcal{O}\left(\sqrt{(1728 - j_0)/3}\right)$ is abelian over $K$, and \cref{cor:index_abelian_twist} shows that $\mathcal{I}(E_{j_0}/\mathbb{Q}(j_0)) = \mathcal{I}(E_{j_0}''/\mathbb{Q}(j_0))$.
Hence, we can conclude that the three families $\mathcal{E}_\text{PARI}, \mathcal{E}_\text{SAGE}$ and $\mathcal{E}_\text{MAGMA}$, when specialised to the same CM $j$-invariant, have the same CM index.
We will use in the rest of the paper, the elliptic curves $E_{j_0}$ obtained by specialising the family $\mathcal{E}_\text{SAGE}$. 
Note that, once the imaginary quadratic order $\mathcal{O}$ is fixed, the index $\mathcal{I}(E_{j_0}/\mathbb{Q}(j_0))$ does not depend on the particular $j$-invariant $j_0 \in \overline{\mathbb{Q}}$ relative to $\mathcal{O}$ to which one specializes the family $\mathcal{E}_\text{SAGE}$, because all these $j$-invariants are conjugate under the action of the absolute Galois group $\operatorname{Gal}(\overline{\mathbb{Q}}/\mathbb{Q})$ (see \cite[Proposition~13.2]{Cox_2013}). 

\begin{algorithm}[b]
    %\centering
    \textbf{Input}: \textsc{Delta} $= \Delta_\mathcal{O}$, the discriminant of $\mathcal{O}$.
    \begin{lstlisting}
from sage.libs.pari.convert_sage import gen_to_sage
R.<x> = PolynomialRing(QQ)
K.<D> = NumberField(x^2-Delta)
F.<j> = K.extension(hilbert_class_polynomial(Delta))
E = EllipticCurve_from_j(j,F)
Fabs.<a> = NumberField(gen_to_sage(pari(F.absolute_polynomial()).polredbest(),{'x' : x}))
Eabs = E.base_extend(F.embeddings(Fabs)[0])
F3.<f3> = Eabs.division_field(3)
F3best.<f3best> = NumberField(gen_to_sage(pari(F3.absolute_polynomial()).polredbest(),{'x' : x}))
F3rel.<f3rel> = F3best.relativize(K.embeddings(F3best)[0])
if F3rel.is_galois_relative() == True:
    Index = gp.rnfisabelian(pari('y^2 + '+str(-Delta)).nfinit(),pari(F3rel.relative_polynomial())) + 1
else:
    Index = 1
\end{lstlisting}
\textbf{Output:} \textsc{Index} $= \mathcal{I}(E_{j_0}/\mathbb{Q}(j_0))$, for any CM $j$-invariant $j_0$ relative to the order $\mathcal{O}$
    \caption{\textsc{SageMath} code to compute the index $\mathcal{I}(E_{j_0}/\mathbb{Q}(j_0))$, relative to the elliptic curve $E_{j_0}$ obtained by specialising the family $\mathcal{E}_\text{SAGE}$ to a CM $j$-invariant $j_0$.}
    \label{alg:sage}
\end{algorithm}

Let us turn now to the computation of the index $\mathcal{I}(E_{j_0}/\mathbb{Q}(j_0))$, where $j_0 \in \overline{\mathbb{Q}}$ is a CM $j$-invariant relative to an order of class number $2$.
The procedure described in \cref{sec:algorithm} simplifies considerably in this case.
Indeed, in general, for any imaginary quadratic order $\mathcal{O} \neq \mathbb{Z}\left[ \frac{1 + \sqrt{-3}}{2} \right]$ and any elliptic curve $E$ with complex multiplication by $\mathcal{O}$ and defined over the ring class field $H_\mathcal{O}$, one has that:
\begin{equation} \label{eq:index_3tors}
    \mathcal{I}(E/H_\mathcal{O}) = \frac{2}{[H_\mathcal{O}(E[3]) \colon H_\mathcal{O}(E[3]) \cap K^\text{ab}]}
\end{equation}
as one can see by combining \cref{prop:finite_computation} and
\cref{prop:3_division_field}.
Moreover, since:
\[
    [H_\mathcal{O}(E[3]) \colon H_\mathcal{O}(E[3]) \cap K^\text{ab}] = \begin{cases}
        1, \ \text{if the extension} \ K \subseteq H_\mathcal{O}(E[3]) \ \text{is abelian;} \\
        2, \ \text{otherwise (as follows from \eqref{eq:index_3tors}, since} \ \mathcal{I}(E/H_\mathcal{O}) \in \mathbb{N}\text{),}
    \end{cases}
\]
we see, using \cref{lem: base change index}, that the computation of $\mathcal{I}(E_{j_0}/\mathbb{Q}(j_0))$ reduces to understanding whether or not the $3$-division field of $E_{j_0}$ is an abelian extension of $K$.
We implemented this computation in \textsc{SageMath} (importing also the functions \href{https://pari.math.u-bordeaux.fr/dochtml/html-stable/General_number_fields.html#polredbest}{\textsc{polredbest}} and \href{https://pari.math.u-bordeaux.fr/dochtml/html-stable/General_number_fields.html#rnfisabelian}{\textsc{rnfisabelian}} from \textsc{Pari/Gp}), as shown in \cref{alg:sage}.
We ran this algorithm for all the $j$-invariants relative to orders $\mathcal{O}$ of class number $2$, whose discriminants $\Delta_\mathcal{O}$ are given by the following list:
\[
    \begin{aligned}
        \Delta_\mathcal{O} \in \{ &-15, -20, -24, -32, -35, -36, -40, -48, -51, -52, -60, -64, -72, -75, -88, -91, \\ 
        &-99, -100, -112, -115, -123, -147, -148, -187, -232, -235, -267, -403, -427 \}
    \end{aligned}
\]
which can be obtained either by applying the algorithms described in \cite{Watkins_2004}, and implemented in SageMath under the function \href{https://doc.sagemath.org/html/en/reference/arithmetic_curves/sage/schemes/elliptic_curves/cm.html#sage.schemes.elliptic_curves.cm.discriminants_with_bounded_class_number}{\textsc{sage.schemes.elliptic\_curves.cm.discriminants\_with\_bounded\_class\_number}}, or by appealing to the classical result \cite[Theorem~1]{Stark_1975}, and then applying the class number formula \cite[Theorem~7.24]{Cox_2013}.
The results of this computation show that $\mathcal{I}(E_{j_0}/\mathbb{Q}(j_0)) = 1$ unless $\Delta_\mathcal{O} = -15$, in which case $\mathcal{I}(E_{j_0}/\mathbb{Q}(j_0)) = 2$.

To conclude, consider the order $\mathcal{O} = \mathbb{Z}[\sqrt{-5}]$ of discriminant $\Delta_\mathcal{O} = -20$, such that $\mathcal{I}(E_{j_0}/\mathbb{Q}(j_0)) = 1$ for every CM $j$-invariant $j_0 \in \overline{\mathbb{Q}}$ relative to $\mathcal{O}$.
We now construct, by a suitable twist of $E := E_{j_0}$ over the Hilbert class field $H := H_\mathcal{O}$, another elliptic curve $E'_{/H}$ with complex multiplication by $\mathcal{O}$, with the property that $\mathcal{I}(E'/H) = 2$.
To do so, we specialize $j_0 = 282880 \sqrt{5} + 632000$, so that $E := E_{j_0}$ is given by:
\begin{equation} \label{eq:Weierstrass_example}
    E \colon y^2 = x^{3} + 29736960 (36023 \sqrt{5} - 80550) x - 55826186240 (16154216 \sqrt{5} + 36121925)
\end{equation}
and we follow the procedure described in the proof of \cite[Theorem~5.11]{Campagna_Pengo_2020}.

More precisely, observe that $H = \mathbb{Q}(\sqrt{-5},i)$ and $3 \cdot \mathcal{O} = \mathfrak{p}_3 \cdot \overline{\mathfrak{p}}_3$ with $\mathfrak{p}_3 = (3,\sqrt{-5} + 1)$ and $\overline{\mathfrak{p}}_3 = (3,\sqrt{-5} - 1 )$. 
By \cite[Theorem~4.6]{Campagna_Pengo_2020} one has that $H_{\mathfrak{p}_3} = H_{\overline{\mathfrak{p}}_3} = H$, where $H_{\mathfrak{p}_3}$ and $H_{\overline{\mathfrak{p}}_3}$ denote respectively the ray class fields of $K$ modulo $\mathfrak{p}_3$ and $\overline{\mathfrak{p}}_3$.
This in particular implies, using \cite[Chapter~II, Theorem~5.6]{si94}, that the $x$-coordinates of the points $P \in E[\mathfrak{p}_3] \cup E[\overline{\mathfrak{p}}_3]$ lie in $H$.
Moreover, it follows from \cite[Lemma~2.4]{Bourdon_Clark_2020} that $\lvert E[\mathfrak{p}_3] \rvert = \lvert E[\overline{\mathfrak{p}}_3] \rvert = 3$, which shows that each non-trivial $\mathfrak{p}_3$-torsion point has the same $x$-coordinate, and similarly for non-trivial $\overline{\mathfrak{p}}_3$-torsion points.
From the factorization: 
\[
    \begin{aligned}
            \phi_{E,3}(x) = 3 \cdot &(x + 594880 + 59840 i - 26048 \sqrt{-5} + 266816 \sqrt{5}) \cdot \\ &(x + 594880 - 59840 i + 26048 \sqrt{-5} + 266816 \sqrt{5} ) \cdot \\ 
            &(x^{2} - (1189760 + 533632 \sqrt{5}) x - 2668089262080 - 1193205432320 \sqrt{5})
    \end{aligned}
\]
of the $3$-division polynomial $\phi_{E,3} \in H[x]$, one can verify that $x_3 := - 594880 - 59840 i + 26048 \sqrt{-5} - 266816 \sqrt{5}$ is the $x$-coordinate of all the non-trivial $\mathfrak{p}_3$-torsion points.
Hence $H(E[\mathfrak{p}_3]) = H(\sqrt{\alpha})$, where:
\[
    \alpha := 13956546560 \cdot (1190435 + 2307955 i - 1032149 \sqrt{-5} + 532379 \sqrt{5})
\]
is obtained by substituting $x_3$ in the right hand side of \eqref{eq:Weierstrass_example}.
It can be checked that the extension $K \subseteq H(\sqrt{\alpha})$ is not Galois, and in particular not abelian, which is compatible with the fact that $\mathcal{I}(E/\mathbb{Q}(j_0)) = 1$.

Thus, the twisted elliptic curve $E' := E^{(\alpha)}$, given by the global minimal Weierstra{\ss} model:
\begin{equation} \label{eq:Weierstrass_Shimura}
    E' \colon y^2 - \left( \frac{1 - i + \sqrt{-5} + \sqrt{5}}{2}\right) x y - \left( \frac{1 + i + \sqrt{-5} + \sqrt{5}}{2} \right) y = x^{3} + x^{2} + \left(2 i - \sqrt{5} \right) x - 1 + 2 i
\end{equation}
has index $\mathcal{I}(E'/H) = 2$, as follows from \eqref{eq:index_3tors}.
Indeed, the first point of \cite[Proposition~5.1]{Campagna_Pengo_2020} implies that $H(E'[\mathfrak{p}_3]) = H_{\mathfrak{p}_3}$, which entails that $H(E'[3])$ coincides with the $3$-ray class field of $K$, as can also be checked by direct computation.
Note finally that $H(E'_\text{tors}) = K^\text{ab}$, as follows from \cref{cor:index_non_0_1728}.

\begin{remark}
The interested reader can find at \cite{Campagna_Pengo_Code} a \textsc{SageMath} notebook in which we implemented the computations carried out to find the elliptic curve $E'$ appearing in \eqref{eq:Weierstrass_Shimura}.
\end{remark}

\section*{Acknowledgements}
We would like to thank François Brunault, Ian Kiming, Fabien Pazuki and Peter Stevenhagen for many useful discussions.
We also thank the anonymous referees for their helpful comments and suggestions.

The first author is supported by ANR-20-CE40-0003 Jinvariant. Moreover, he wishes to thank the Max Planck Institute for Mathematics in Bonn for
its financial support, great work conditions and an inspiring atmosphere. 

The second author performed this work within the framework of the LABEX MILYON (ANR-10-LABX-0070) of Universit{\'e} de Lyon, within the program ``Investissements d'Avenir'' (ANR-11-IDEX-0007) operated by the French National Research Agency (ANR).

\printbibliography

\end{document}

%% file: shortcuts.tex
\DeclareMathOperator{\Gal}{Gal}

\DeclareMathOperator{\End}{End}
\DeclareMathOperator{\Aut}{Aut}

\newcommand\restr[2]{{ \left.\kern-\nulldelimiterspace #1 \vphantom{\big|} \right|_{#2} }}

\newcommand\altxrightarrow[2][0pt]{\mathrel{\ensurestackMath{\stackengine%
  {\dimexpr#1-5pt}{\xrightarrow{\phantom{#2}}}{\scriptstyle\!#2\,}%
  {O}{c}{F}{F}{S}}}}